\documentclass[a4paper,12pt,oneside,reqno]{amsart}
\usepackage{amssymb,amsfonts,amsmath,amsthm,amscd, bbm}
\usepackage{graphicx, color}
\usepackage{mathrsfs}
\numberwithin{equation}{section}

\newcommand{\beq}{\begin{equation}}
\newcommand{\eeq}{\end{equation}}
\newcommand{\bea}{\begin{aligned}}
\newcommand{\eea}{\end{aligned}}
\newcommand{\bdm}{\begin{displaymath}}
\newcommand{\edm}{\end{displaymath}}
\newcommand{\barr}{\begin{array}}
\newcommand{\earr}{\end{array}}
\newcommand{\ben}{\begin{enumerate}}
\newcommand{\een}{\end{enumerate}}
\newcommand{\bde}{\begin{description}}
\newcommand{\ede}{\end{description}}

\newtheorem{teor}{Theorem}
\newtheorem{prop}[teor]{Proposition}
\newtheorem{lem}[teor]{Lemma}
\newtheorem{cor}[teor]{Corollary}
\newtheorem{fact}{Fact}

\newtheorem{rem}[teor]{Remark}
\newcommand{\R}{\mathbb{R}}

\newcommand{\N}{\mathbb{N}}

\newcommand{\PP}{\mathbb{P}}
\newcommand{\E}{{\mathbb{E}}}

\newcommand{\defi}{\equiv} 

\newcommand{\de}{\delta}

\newcommand{\e}{\epsilon}
\newcommand{\la}{\lambda}
\newcommand{\w}{\omega}

\newcommand{\vare}{\varepsilon}

\newcommand{\F}{\mathcal{F}}
\newcommand{\1}{\mathbbm{1}}

\DeclareMathOperator{\Poi}{Poi}
\newcommand\blfootnote[1]{%
  \begingroup
  \renewcommand\thefootnote{}\footnote{#1}%
  \addtocounter{footnote}{-1}%
  \endgroup
}
\begin{document}

\title[FPP in the hypercube]
{Oriented first passage percolation in the mean field limit, 2. The extremal process.}

\author[N. Kistler]{Nicola Kistler}
\address{Nicola Kistler \\ J.W. Goethe-Universit\"at Frankfurt, Germany.}
\email{kistler@math.uni-frankfurt.de}

\author[A. Schertzer]{Adrien Schertzer}
\address{adrien schertzer \\ J.W. Goethe-Universit\"at Frankfurt, Germany.}
\email{schertzer@math.uni-frankfurt.de}

\author[M. A. Schmidt]{Marius A. Schmidt}
\address{Marius A. Schmidt \\ J.W. Goethe-Universit\"at Frankfurt, Germany.}
\email{mschmidt@math.uni-frankfurt.de}

 \date{\today}

\begin{abstract}
This is the second, and last paper in which we address the behavior of oriented first passage percolation on the hypercube in the limit of large dimensions. We prove here that the extremal process converges to a Cox process with exponential intensity. This entails, in particular, that the first passage time converges weakly to a random shift of the Gumbel distribution. The random shift, which has an explicit, universal distribution related to modified Bessel functions of the second kind, is the sole manifestation of correlations ensuing from the geometry of Euclidean space in infinite dimensions. The proof combines the multiscale refinement of the second moment method with a conditional version of the Chen-Stein bounds, and a contraction principle.
\end{abstract}
\blfootnote{\textup{2000} \textit{Mathematics Subject Classification}: \textup{60J80, 60G70, 82B44}} 

\maketitle


\section{Introduction and main results}
The model we consider is constructed as follows. We first embed the $n$-dimensional hypercube in $\R^n$: for $e_1,..,e_n$ the standard basis, we identify the hypercube as the graph $G_n \defi (V_n, E_n)$, where $V_n= \{0,1\}^n$ and $E_n \defi \{(v,v+e_j): v,v+e_j \in V, j\leq n\}$.  The set of shortest (directed) paths connecting diametrically opposite vertices, say $\boldsymbol{0}\defi (0,..,0)$ and $\boldsymbol{1}\defi (1,..,1)$, is   given by
\beq
\Sigma_n \defi \{ \pi\in V_{n+1}: \pi_1 = \boldsymbol{0}, \pi_{n+1} = \boldsymbol{1}, (\pi_i,\pi_{i+1})\in E_n, \forall i\leq n\}.
\eeq
A graphical rendition is given in Figure \ref{cube_paths} below.
\begin{figure}[h!]
\includegraphics[scale=0.33]{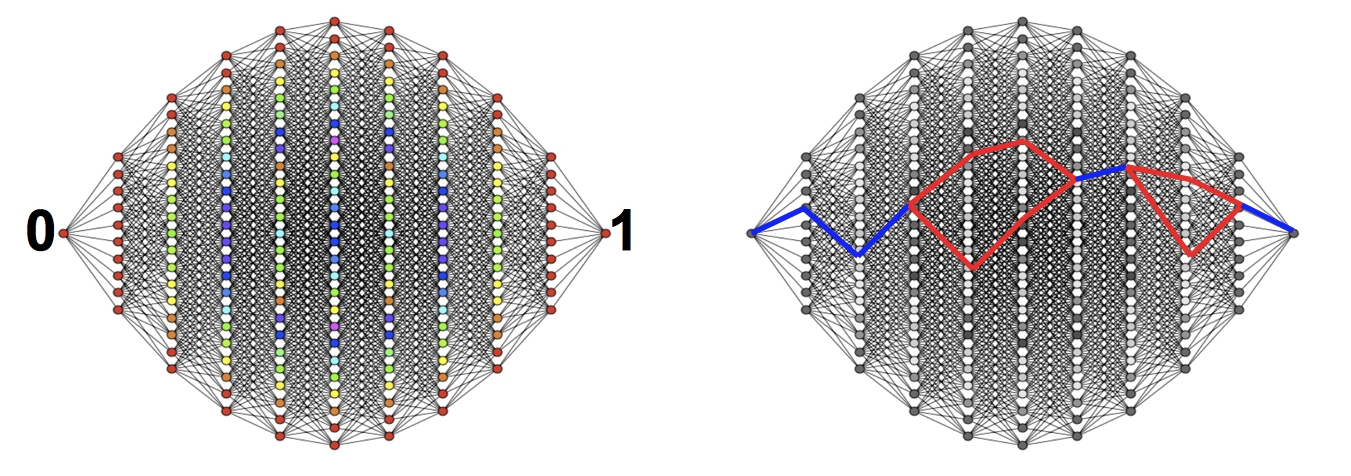}
\caption{The 10-dimensional hypercube (left), and two oriented connecting paths (right): blue edges are common to both paths, whereas paths do not overlap on red edges.}\label{cube_paths}
\end{figure}

Let now $(\xi_{e})_{e\in E}$ be a family of independent standard exponentials, i.e. exponentially distributed random variables with parameter 1, and assign to  each oriented path $\pi \in \Sigma_n$ its {\it weight}
$$X_\pi\defi\sum_{k\leq n} \xi_{[\pi]_k},$$
where $[\pi]_i = (\pi_i,\pi_{i+1})$ is the $i$-th edge of the path. 

A key question in first passage percolation, FPP for short, concerns the so-called  {\it first passage time},
\beq \label{one}
m_n \defi \min_{\pi \in \Sigma_n}X_\pi \,,
\eeq
namely the smallest weight of connecting paths. The limiting value of $m_n$ {\it to leading order} has been settled by Fill and Pemantle \cite{Fill_Pemantle}, who proved that
\beq \label{lim_leading}
\lim_{n\to \infty} m_n = 1,
\eeq
almost surely. 

The "law of large numbers" \eqref{lim_leading} naturally raises questions on fluctuations and weak limits, and calls for a description of the paths with minimal weight. As a first step towards this goal we presented in \cite{kss} an alternative, "modern" approach to \eqref{lim_leading} much inspired by the recent advances in the study of Derrida's random energy models (see \cite{kistler} and references therein) and which relies on the hierarchical approximation to the FPP.
In this companion paper we bring the approach to completion by establishing the full limiting picture, i.e. identifying the weak limit of the {\it extremal process}
\[
\Xi_n \defi \sum\limits_{\pi\in \Sigma_n} \delta_{n(X_\pi-1)}\,.
\]

\begin{teor}[Extremal process] \label{pf} Let $\Xi$ be a Cox process with intensity $ Z e^{x-1} dx$, where $Z$ is distributed like the product of two independent standard exponentials. Then
\beq \label{pf_lim}
\lim_{n\to \infty} \Xi_n  = \Xi,
\eeq
weakly. In particular, it follows for the first passage time $m_n$ that
\beq
\lim_{n\to \infty} \PP(   n(m_n-1) \leq t) = \int\limits_{0}^{\infty} \frac{x}{e^{1-t}+x}e^{-x} dx\,.
\eeq
\end{teor}

It will become clear in the course of the proof, see in particular Remark \ref{universality} below, that the assumption on the distribution of the edge-weights is no restriction: any distribution in the extremality class of the exponentials (i.e. any distribution with similar behavior for small values, to leading order) will lead to the same limiting picture and weak limits. 
Although not needed, we also point out that the distribution of the mixture is given by $f(z) = 2 z^2 K_0(2 \sqrt{z})$, with $K_0$ a modified Bessel function of the second kind.

What lies behind the onset of the Cox processes is a {\it decoupling} whose origin can be traced back to the high-dimensional nature of the problem at hand. Indeed, the following mechanism, depicted in Figure \ref{paths_lim} below, holds with overwhelming probability in the limit $n\to \infty$ first, and $r \to \infty$ next: {\it Walkers connecting {\bf 0} to {\bf 1} through paths of minimal weight may share at most the first $r$ steps of their journey. Yet, and crucially: whenever they depart from one another ('branch off'), they cannot meet again until they lie at distance at most  $r$ from the target. If meeting happens, they must continue on the same path (no further branching is possible). } The long stretches during which optimal paths do not overlap are eventually responsible for the Poissonian component of the extremal process, whereas the mixing is due to the relatively short stretches of  tree-like (early and late) evolution of which the system keeps persistent memory. The picture is thus very reminiscent of the extremes of branching Brownian motion [BBM], see \cite{bovier} and references therein. More specifically, the extremal process of FPP on the hypercube can be (partly) seen as the "gluing together" of two extremal processes of BBM in the weak correlation regime as studied by Bovier and Hartung \cite{bovier_hartung, bovier_hartung_1}, see also \cite{derrida_spohn, fang_oz, fang_oz_1}.

\begin{figure}[h!] \label{paths_lim}
\includegraphics[scale=0.22]{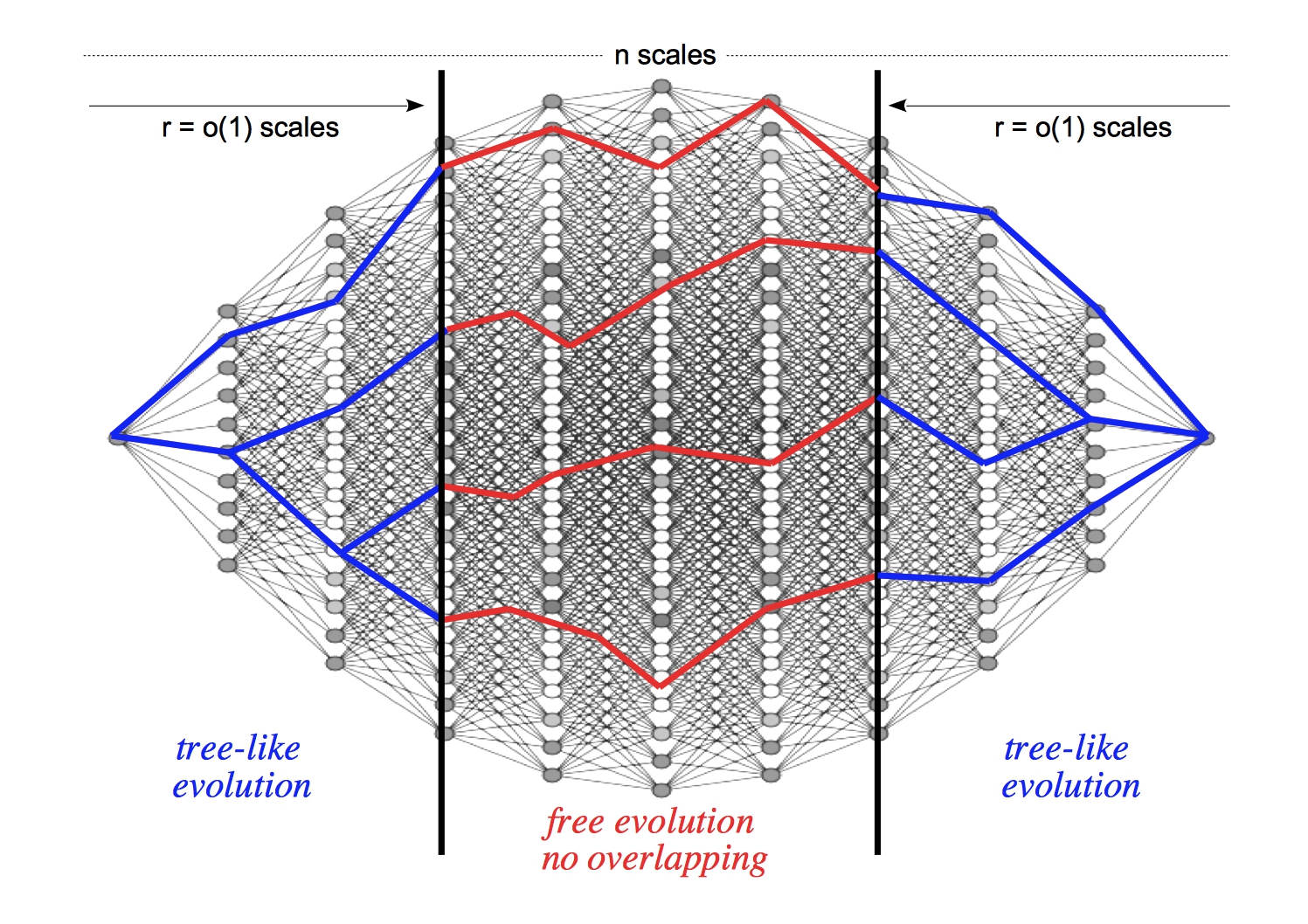}
\caption{Four extremal paths. Remark in particular the tree-like evolution close to {\bf 0} and {\bf 1} (blue edges) and the (comparatively) longer stretch where paths share no common edge (red). This should be contrasted with the low-dimensional scenario: "loops" in the core of the hypercube, as depicted  in Figure \ref{cube_paths}, become less and less likely as the dimension grows. }\label{paths_lim}
\end{figure}
${}$

{\bf Acknowledgements.} It is our pleasure to thank Ralph Neininger for much needed guidance in the field of contraction methods and distributional fixed points.

\section{Strategy of proof} \label{road}
The approach amounts to exploiting the insights on the physical mechanisms summarized in Figure \ref{paths_lim}. Specifically, we will check convergence of intensity and avoidance functions of the extremal process. To see how this comes about, we lighten notation by setting, for $A\subset \R$ a generic subset and $\pi$ an oriented path,
\[
 I_\pi(A) \defi \delta_{n(X_\pi-1)}(A), \quad \text{and}\quad \Xi_n(A) \defi \sum_{\pi \in \Sigma_n} I_\pi(A)\,.
\]
We then claim that with $Z$ as in Theorem \ref{pf}, and $A$ a finite union of bounded intervals:
\begin{itemize}
\item Convergence of the intensity:
\beq \label{intens}
\lim\E\, \Xi_n(A) \underset{n \to \infty}{\longrightarrow}  \E \int_A Z e^{x-1} dx
= \int_A e^{x-1} dx \,.
\eeq
\item Convergence of the avoidance function:
\beq  \label{avoid}
\PP\left( \Xi_n(A) = 0\right) \underset{n \to \infty}{\longrightarrow}  \PP\left(\Xi(A) = 0\right)= \E\left[\exp\left(-Z \int_{A} e^{x-1} dx\right)\right].
\eeq
\end{itemize}
Theorem  \ref{pf} then immediately follows in virtue of Kallenberg's Theorem \cite[Theorem 4.15]{kallenberg}. The proof of the claim on the intensity is rather straightforward: it only requires tail-estimates which we now state for they will be constantly used throughout the paper. (The simple proof may be found in \cite[Lemma 5]{kss}).

\begin{lem} \label{tail} Let $\{\xi_i\}_{i\leq n}$ be  independent standard exponentials, and set $X_n \defi \sum_{i=1}^n \xi_i$. Then
\beq \label{law_i}
\PP\left( X_n\leq x\right) = \left(1+K(x,n) \right)\frac{e^{- x} x^n}{n!},
\eeq
for $x>0$ and with the error-term satisfying $0\leq K(x,n)\leq e^{ x} x/(n+1).$  \\
\end{lem}
Armed with these estimates we can proceed to the short proof of \eqref{intens}. Here and below, we will always consider sets of the form $A = \left(-\infty,a\right]$ , $a\in \R$. This is enough for our purposes since the general case follows by additivity. It holds:
\beq \bea
\E \Xi_n(A) &= \sum\limits_{\pi \in \Sigma_n} \PP\left( n(X_\pi - 1) \leq a \right) \\
&= n!\PP\left( n(X_{\pi^{*}} - 1) \leq a \right) \qquad \text{(symmetry, $\pi^{*} \in \Sigma_n$ is arbitrary)} \\
&= n! \left\{ 1+K\left(1+\frac{a}{n},n \right) \right\} e^{- 1-\frac{a}{n}} \left(\left(1+\frac{a}{n}\right)^+\right)^{n} (n!)^{-1} \qquad \text{(Lemma \ref{tail})}\\
&=(1+o_n(1))e^{- 1+a} \\
& = (1+o_n(1)) \int_A e^{x-1} dx,
\eea \eeq
as claimed.  Convergence of the intensity \eqref{intens}  is thus already settled.  \\

\noindent Contrary to convergence of the intensity, convergence of avoidance functions \eqref{avoid} will require a fair amount of work. This will be split in a number of intermediate steps. The main ingredient is a conditional version of the Chen-Stein bounds:

\begin{teor} [Conditional Chen-Stein Method]\label{Cond_ChenStein}
Consider a probability space $(\Omega,\mathscr{F},\PP)$, a sigma-algebra $\F \subset \mathscr{F}$, a finite set $I$, and a family $(X_i)_{i\in I}$ of Bernoulli random variables issued on this space. Let furthermore
\[
W=\sum_{i\in I} X_i \quad \text{and} \quad \lambda=\sum_{i\in I}\E(X_i |\F )\,.
\]
Finally, consider a random variable $\widehat{W}$ with the property that its law conditionally upon $\F$ is Poisson, i.e.  $\mathcal{L}(\widehat{W}|\F)=\Poi(\lambda)$. It then holds:
\beq \label{ccs}
d_{TV|\F}(W,\widehat{W})\leq \sum_{i\in I}\E(X_i|\F )^2+\sum_{i\in I}\sum_{j\in N_i}(\E(X_i|\F )\E(X_j|\F )+\E(X_iX_j|\F ))\,,
\eeq
where
\[
d_{TV|\F}(W,\widehat{W}) \defi \sup_{A \in \mathscr{F}}\left(\PP_W(A|\F)-\PP_{\widehat{W}}(A|\F)\right)
\]
is the total variation distance conditionally upon $\F$. Finally, $N_i, i \in I$ is a collection of conditionally dissociating neighborhoods, i.e. with the property that $X_i$ and $ \{ X_j : j\in (N_i\cup \{i\})^c \}$ are independent, conditionally upon $\F$.
\end{teor}

Theorem \ref{Cond_ChenStein} is a variant of the classical Chen-Stein method which is tailor-suited to our purposes. Since we haven't found in the literature any similar statement, we provide the rather short proof in the appendix for completeness.

In order to prove convergence of the avoidance functions, we will apply Theorem \ref{Cond_ChenStein} by  conditioning on the left- and rightmost regions of Figure \ref{paths_lim}, namely those regions where tree-like evolutions eventually kick in.  Specifically, in order to apply the conditional Chen-Stein, we make the following choices:

\begin{itemize}
\item[a)] $I  \defi \Sigma_n$, the set of admissible (oriented) paths connecting {\bf 0} to {\bf 1 }.
\item[b)] $\F$ is the sigma-algebra generated by the weights of edges at distance at most $r$ from $\boldsymbol{0}$ or $\boldsymbol{1}$, to wit
\[
\F = \mathcal{F}_{r,n} \defi \sigma( \xi_{e}: e=(u,v)\in E, \min \{d(u,\boldsymbol{0}),d(v,\boldsymbol{0})\}\in \left[0,r\right) \cup \left[n-r,n\right) )\,.
\]
\item[c)] The family of Bernoulli r.v.'s is given by $\left( I_\pi(A) \right)_{\pi \in \Sigma_n}$.
\item[d)] The (random) Poisson-parameter is
\[
\la = \la_{r,n}(A) \defi  \sum\limits_{\pi\in \Sigma_n} \E\left[ I_{\pi}(A)\mid \mathcal{F}_{r,n} \right]
\]
\item[e)] The dissociating neighborhoods are given, for $\pi \in\Sigma_n$, by
\[
N_\pi \defi \{\pi'\in \Sigma_n \setminus\{\pi\}: \exists i\in \{r+1,..,n-r\} \mbox{ s.t. } [\pi]_i = [\pi']_i\}
\]
\end{itemize}
A first, fundamental observation concerns item d), namely the weak convergence of the Poisson-parameter in the double limit $n\to \infty$ first and $r \to \infty$ next. This is an instructive warm-up computation which we now explain. 

Denote the set of all pairs of paths leading $r$-steps away from the start/end respectively, and which can be part of an oriented path from $\boldsymbol{0}$ to $\boldsymbol{1}$ by
\beq
\bea
\mathcal{V}_{r,n}=&\{(x,y)\in V^{r+1}\times V^{r+1}: x_1 = \boldsymbol{0}, d(x_{r+1},\boldsymbol{0}) = r, d(y_{1},\boldsymbol{1}) = r, y_{r+1} = \boldsymbol{1},\\
& \hspace{4.5cm} y_1-x_{r+1}\in V, (x_i,x_{i+1}),(y_i,y_{i+1})\in E, \forall i\leq r\}.
\eea
\eeq
Note that $y_1-x_{r+1}\in V$ is equivalent to there being a directed path from $\boldsymbol{0}$ to $\boldsymbol{1}$ containing $x$ and $y$.   For $(x,y) \in \mathcal{V}_{r,n}$ we define the set of paths connecting $x$ and $y$ by
\beq\bea \Sigma_{x,y} \defi &\{ \pi' \in V^{n-2r+1}: \exists \pi\in \Sigma_n \mbox{ s.t. } ([\pi]_{i})_{i\leq r} = ([x]_i)_{i\leq r}\quad \text{and}  \\
& \qquad \quad ([\pi]_{i})_{r<i\leq  n-r} = ([\pi']_i)_{r< i \leq  n-r}, ([\pi]_{i})_{ i > n-r} = ([y]_i)_{i > n-r}\}. \eea\eeq
By definition,
\beq \bea \label{la_one}
\la_{r, n}(A) = & \sum\limits_{\pi\in \Sigma_n}\PP\left(n(X_\pi- 1) \leq a\Big| \mathcal{F}_{r,n} \right)  \\
&= \sum\limits_{(x,y) \in \mathcal{V}_{r,n}}\sum\limits_{\pi'\in \Sigma_{x,y}}\PP\left(\sum\limits_{i=1}^{n-2r}\xi_{[\pi']_i} \leq 1+ \frac{a}{n}-\sum\limits_{i=1}^{r}\xi_{[x]_i} +\xi_{[y]_i}  \Big| \mathcal{F}_{r,n} \right). \eea\eeq
Shorten $$X_{x,y} \defi\sum\limits_{i=1}^{r}\xi_{[x]_i} +\xi_{[y]_i} .$$ By Lemma \ref{tail}, and since $|\Sigma_{x,y}| = (n-2r)!$, the r.h.s. of  \eqref{la_one} equals
\beq\label{cond_exp_eq1}
\sum\limits_{(x,y) \in \mathcal{V}_{r,n}} \left(1+K(1+ \frac{a}{n}-X_{x,y},n-2r) \right)\exp\left( -1- \frac{a}{n}+X_{x,y}\right)\left( \left(1+ \frac{a}{n}-X_{x,y}\right)^{+}\right)^{n-2r}.
\eeq
By the tail-estimates from Lemma \ref{tail}, the following holds
\[
K\left(1+ \frac{a}{n}-X_{x,y},n-2r \right)  \leq \frac{2e^2}{n-2r}\,,
 \]
for all non-zero summands, and $n\geq a$. Remark that there are  $O(n^{2r})$ such summands, while $r$ and $a$ are fixed: one easily checks that dropping all summands where $X_{x,y}> (\ln n)^2/n$ only causes a deterministically vanishing error, hence
\beq \bea \label{cond_exp_eq2}
\eqref{cond_exp_eq1} &= (1+o_n(1)) \left( o_n(1)+ e^{-1}\sum\limits_{(x,y) \in \mathcal{V}_{r,n}} \mathbbm{1}_{\{X_{x,y} \leq  \frac{(\ln n)^2}{n}\}} \exp\left( (n-2r)\ln\left(1+ \frac{a}{n}-X_{x,y}\right)^{+}\right)\right)\\
&=(1+o_n(1))\left( o_n(1)+ e^{-1+a}\sum\limits_{(x,y) \in \mathcal{V}_{r,n}} \exp\left( - n X_{x,y}\right)\right) \\
& = (1+o_n(1))\left( o_n(1)+ e^{-1+a}\sum\limits_{(x,y) \in \mathcal{V}_{r,n}} \exp - n \sum_{i=1}^{r} \left(  \xi_{[x]_i} +\xi_{[y]_i}  \right) \right)\,.
\eea\eeq
the second step by Taylor-expanding the logarithm around $1$ to first order, and the third by definition. \\

We now address the sum on the r.h.s. of \eqref{cond_exp_eq2}, on which we perform the aforementioned double limit $n\to \infty$ first and $r\to \infty$ next. The upshot is summarized in Proposition \ref{double_limit} below, whose proof --- via a contraction argument --- is deferred to Section \ref{doub_weak}. To formulate, we need some additional notation: for $\pi_1,..,\pi_{i-1}\in \N$ and $i\leq r$ we denote by
\[
\left(\eta_{\pi_1,..,\pi_{i-1}, \pi_i} \right)_{\pi_i \in \N}, \quad \text{and}\quad \left(\tilde{\eta}_{\pi_1,..,\pi_{i-1}, \pi_i}\right)_{\pi_i\in \N}
\]
independent Poisson point processes [PPP] with intensity $\mathbbm{1}_{\R^+}dx$, and set
\beq \label{z}
Z_{r} \defi  \sum_{\pi \in \N^r}\exp\left(- \sum_{j=1}^{r}\eta_{\pi_1\pi_2...\pi_j}\right), \quad \widetilde Z_{r} \defi  \sum_{\pi \in \N^r}\exp\left(- \sum_{j=1}^{r} \tilde \eta_{\pi_1\pi_2...\pi_j}\right)\,.
\eeq

\begin{prop}  \label{double_limit} (The double weak-limit).
\begin{itemize}
\item \emph{n-convergence}: the following weak limit, to fixed $r$, holds:
\[
\lim_{n\to \infty} \sum\limits_{(x,y) \in \mathcal{V}_{r,n}} \exp -n \sum\limits_{l=1}^r \left( \xi_{[x]_l}+\xi_{[y]_l} \right) =  Z_r \times \widetilde Z_r\,.
\]
\item \emph{r-convergence}: $Z_r$ and $\widetilde Z_r$ weakly converge, as $r\to \infty$, to independent standard exponentials.
\end{itemize}
\end{prop}

The n-convergence is a key ingredient in Figure \ref{paths_lim} above. Indeed, remark that both limits $Z_r$ and $\widetilde Z_r$ are constructed outgoing from {\it hierarchical}\footnote{Superpositions of PPP such as those involved in \eqref{z} are ubiquitous in the Parisi theory of mean field sping glasses, see  \cite{kistler} and references, where they are referred to as \emph{Derrida-Ruelle cascades}. Although no knowledge of the Parisi theory is assumed/needed, our approach to the oriented FPP in the limit of large dimensions heavily draws on ideas which have recently crystallised in that field.} superpositions of PPP: this accounts for the somewhat surprising fact that close to {\bf 0} and {\bf 1} only tree-like structures contribute to the extremal process in the mean field limit. \\

Proposition \ref{double_limit} and \eqref{cond_exp_eq2} steadily imply convegence of the Poisson-parameter:

\begin{cor} \label{conv_lambda}
With the above notations,
\[
\lim_{r \to \infty} \lim_{n \to \infty} \la_{r,n}(A) = Z \int_A e^{x-1} dx,
\]
weakly.
\end{cor}
We now come back to the main task of proving \eqref{avoid}, convergence of the avoidance functions. The line of reasoning goes as follows: recalling that $\Xi_n(A) = \sum_{\pi \in \Sigma_n} I_{\pi}(A)$, we write

\beq \bea \label{flow_one}
& \left| \PP\left( \Xi_n(A) = 0 \right) - \PP\left( \Xi(A) = 0 \right) \right| \\
& \qquad = \left| \E \PP\left( \Xi_n(A) = 0 \mid \F_{r,n} \right) - \E \PP\left( \Xi(A) = 0 \mid Z \right) \right| \\
& \qquad \qquad \leq \left| \E \PP\left( \Xi_n(A) = 0 \mid \F_{r,n} \right) -  \PP\left( \Poi\left(\la_{r,n}(A) \right)=0 \mid \F_{r,n} \right) \right|  \\
& \qquad \qquad \qquad + \left|  \E \PP\left( \Poi\left(\la_{r,n}(A) \right)=0 \mid \F_{r,n} \right)- \E \PP\left( \Xi(A) = 0 \mid Z \right) \right|,
\eea \eeq
by the triangle inequality. By convexity, one has  
\beq \bea \label{flow_two}
& \left| \E \PP\left( \Xi_n(A) = 0 \mid \F_{r,n} \right) -  \PP\left( \Poi\left(\la_{r,n}(A) \right)=0 \mid \F_{r,n} \right) \right|  \\
& \qquad \leq\E   \left|\PP\left( \Xi_n(A) = 0 \mid \F_{r,n} \right) - \PP\left( \Poi\left(\la_{r,n}(A) \right)=0 \mid \F_{r,n} \right) \right|  \\
& \qquad \leq \E d_{TV,\mathcal{F}_{r,n}} \left( \Xi_n(A), \Poi\left(\la_{r,n}(A) \right) \right) \\
& \qquad =: \text{CS}(r,n),  \quad \text{say}. \\
\eea \eeq
Furthemore, by definition
\beq \bea \label{flow_tree}
& \left|  \E \PP\left( \Poi\left(\la_{r,n}(A) \right)=0 \mid \F_{r,n} \right)- \E \PP\left( \Xi(A) = 0 \mid Z \right) \right| \\
& =\left| \E \left(e^{-\la_{r,n}(A)}-e^{-Z \int_A e^{x-1} dx}\right) \right| \\
& =: \text{P}(r,n), \quad \text{say}.
\eea \eeq
It thus follows from \eqref{flow_one},  \eqref{flow_two} and  \eqref{flow_tree} that
\beq \bea \label{flow_four}
& \left| \PP\left( \Xi_n(A) = 0 \right) - \PP\left( \Xi(A) = 0 \right) \right| \leq \text{CS}(r,n) + \text{P}(r,n).
\eea \eeq

The second term is easily seen to vanish thanks to the convergence of the Poisson-parameter: it follows from Corollary \ref{conv_lambda} and weak limit that 

\beq \label{par}
\lim_{r \to \infty} \lim_{n\to \infty} \text{P}(r, n) = 0.
\eeq
We finally claim that the first term in \eqref{flow_four}, the "Chen-Stein term", also vanishes in the considered double-limit, to wit:
\beq \label{cs_vanishes}
\lim_{r \to \infty} \lim_{n\to \infty} {\text{CS}}(r,n) = 0\,.
\eeq
This claim is proved in Section \ref{vanishing_cs} as an application of the conditional Chen-Stein method.

Combining \eqref{par} and \eqref{cs_vanishes} we thus obtain convergence of the avoidance function: since this was the last missing ingredient, our main Theorem \ref{pf} follows.

\section{Proofs}

\subsection{The double weak-limit} \label{doub_weak}
The goal of this section is to prove Proposition \ref{double_limit}. We first address the n-convergence, which states
that
\beq \label{conv_first}
\lim_{n\to \infty} \sum\limits_{(x,y) \in \mathcal{V}_{r,n}} \exp\left( -n \sum\limits_{l=1}^r \xi_{[x]_l}+\xi_{[y]_l} \right) = Z_r \times \widetilde Z_r,
\eeq
weakly, where $Z_r, \widetilde Z_r$ are defined in \eqref{z}. The idea here is to enlarge the set of paths over which the sum is taken, as this enables a useful decoupling. Precisely, consider the set of directed paths of length $r$ {\it from $\boldsymbol{0}$},
\beq
\mathcal{V}^{\leftarrow}_{r,n}=\{x\in V^{r+1}: x_1 = \boldsymbol{0}, d(x_{r+1},\boldsymbol{0}) = r, [x]_i\in E, \forall i\leq r\}\,,
\eeq
and respectively {\it to} $\boldsymbol{1}$:
\beq
\mathcal{V}^{\rightarrow}_{r,n}=\{y\in V^{r+1}: y_{r+1} = \boldsymbol{1}, d(y_{1},\boldsymbol{1}) = r,[y]_i \in E, \forall i\leq r\}\,.
\eeq
One easily checks that
\beq \label{card_diff}
\left| \mathcal{V}^{\leftarrow}_{r,n}\times \mathcal{V}^{\rightarrow}_{r,n} \setminus \mathcal V_{r,n}\right| = O(n^{2r-1})\,.
\eeq
We split the sum over the larger subset into a sum over $\mathcal V_{r,n}$ and a "rest-term":
\beq \bea
& \sum\limits_{(x,y) \in \mathcal{V}^{\rightarrow}_{r,n}\times\mathcal{V}^{\leftarrow}_{r,n}  } \exp\left( -n \sum\limits_{l=1}^r \xi_{[x]_l}+\xi_{[y]_l} \right)= \\
& \hspace{2cm} = \sum\limits_{(x,y) \in \mathcal V_{r,n} } \exp\left( -n \sum\limits_{l=1}^r \xi_{[x]_l}+\xi_{[y]_l} \right) \\
& \hspace{4cm}+\sum\limits_{(x,y) \in  (\mathcal{V}^{\rightarrow}_{r,n}\times\mathcal{V}^{\leftarrow}_{r,n}  )\setminus\mathcal{V}_{r,n}   } \exp\left( -n \sum\limits_{l=1}^r \xi_{[x]_l}+\xi_{[y]_l} \right)\,.
\eea \eeq
and claim that the term on the r.h.s. vanishes in probability. Indeed, by a simple computation involving the moment generating function of the exponential distribution,
\beq \bea
\E\left| \sum\limits_{(x,y) \in (\mathcal{V}^{\rightarrow}_{r,n}\times\mathcal{V}^{\leftarrow}_{r,n}  )\setminus\mathcal{V}_{r,n}}  \exp\left( -n \sum\limits_{l=1}^r \xi_{[x]_l}+\xi_{[y]_l} \right) \right| &=  \left| (\mathcal{V}^{\rightarrow}_{r,n}\times\mathcal{V}^{\leftarrow}_{r,n}  )\setminus\mathcal{V}_{r,n}\right| (n+1)^{-2r} \\
& \stackrel{\eqref{card_diff}}{\longrightarrow} 0, \quad n\to \infty.
\eea \eeq
It thus follows from Markov's inequality that the contribution of paths in $(\mathcal{V}^{\rightarrow}_{r,n}\times\mathcal{V}^{\leftarrow}_{r,n}  )\setminus\mathcal{V}_{r,n}$ is irrelevant for our purposes:  the weak limit when summing over $\mathcal{V}_{r,n}$, and that when summing over $\mathcal{V}^{\rightarrow}_{r,n}\times \mathcal{V}^{\leftarrow}_{r,n}$ coincide, provided one of them exists. On the other hand, the sum over the enlarged set of paths "decouples" into two independent identically distributed terms:
\beq \sum\limits_{(x,y) \in \mathcal{V}^{\rightarrow}_{r,n}\times\mathcal{V}^{\leftarrow}_{r,n} } \exp\left( -n \sum\limits_{l=1}^r \xi_{[x]_l}+\xi_{[y]_l} \right)=\sum\limits_{x \in \mathcal{V}^{\rightarrow}_{r,n} } \exp\left( -n \sum\limits_{l=1}^r \xi_{[x]_l}\right)\sum\limits_{y \in \mathcal{V}^{\leftarrow}_{r,n} } \exp\left( -n \sum\limits_{l=1}^r \xi_{[y]_l} \right).\eeq
The n-convergence will therefore follow as soon as we show that
\beq  Z_{r,n}\defi \sum\limits_{x \in \mathcal{V}^{\rightarrow}_{r,n} }\exp\left( -n \sum\limits_{l=1}^r \xi_{[x]_l}\right) \underset{n \to \infty}{\longrightarrow}\sum_{\pi \in \N^r}\exp\left(\sum_{l=1}^{r}-\eta_{\pi_1\pi_2...\pi_j}\right) \defi Z_r\eeq
holds weakly. This will be done by induction on $r$. The base-case $r=1$ is addressed in
\begin{lem} \label{init}
Consider $\eta \defi  \sum_{i \in \N} \de_{\eta_{i}} $ a \emph{PPP}($\mathbbm{1}_{\R^+}dx$) and independent standard exponentials $(\xi_{i})_{i\in \N}$. It then holds:
\beq\bea
\sum\limits_{i=1}^{n} \delta_{\xi_{i}n}  \underset{n \to \infty}{\longrightarrow} \eta \label{evt}
\eea \eeq
weakly. Furthermore, the following weak limit holds:
\beq\bea
\sum\limits_{i=1}^{n} \exp\left( - \xi_{i}n\right)  \underset{n \to \infty}{\longrightarrow} \sum\limits_{i\in \N} \exp\left( - \eta_{i}\right)\,. \label{functional}
\eea \eeq
\end{lem}
\begin{rem} \label{universality}
In virtue of Lemma \ref{init}, Theorem \ref{pf} holds for any choice of edge-weights falling in the universality class of the exponential distribution, i.e. for which \eqref{evt} holds.
\end{rem}

\begin{proof}[Proof of Lemma \ref{init}]
Claim \eqref{evt} is a classical result in extreme value theory. We omit the elementary proof.
As for the second claim: it is steadily checked (e.g. by Markov's inequality) that the sum  on the l.h.s. of \eqref{functional} is almost surely finite. In order to prove \eqref{functional} it thus suffices to
compute the Laplace transform of the two sums. For $t \in \R^{+}$, since the $\xi's$ are independent, we have:
\beq\bea \label{boh}
\E\exp - t \sum_{i=1}^{n} e^{- \xi_{i}n } = {\E\left( e^{te^{-\xi_{1}n}}\right)}^n & ={\left(1+\int\limits_{0}^{+\infty}e^{-x}(e^{te^{-xn}}-1)dx\right)}^n\\
&={\left(1+\frac{1}{n} \int\limits_{0}^{+\infty}e^{-u/n}(e^{te^{-u}}-1)du \right)}^n \,,
\eea \eeq
the second equality by change of variable. But $e^{-u/n}(e^{te^{-u}}-1)\leq (e^{te^{-u}}-1)$, which is integrable, hence  by dominated convergence we have that the r.h.s. of \eqref{boh} converges, as $n\uparrow \infty$, to the limit
\beq\bea
\exp\left(\int\limits_{0}^{+\infty} (e^{ -te^{-x}}-1)dx\right) =  \E \exp -t\sum\limits_{i\in \N} e^{- \eta_{i}},
\eea \eeq
where the last equality follows by a simple computation: \eqref{functional} is therefore settled.
\end{proof}

\noindent For the n-convergence, we will work with the Prohorov metric, which we recall is defined as follows: for $\mu, \nu \in \mathcal M_1(\R)$ two probability measures, the Prohorov distance is given by
\[
\text{d}_{\text{P}}\left(\mu, \nu \right) \defi \inf\left\{ \e> 0:  \mu(A) \leq \nu(A^\e) + \e,\; \forall A \subset \R \, \text{closed}\right\}\,,
\]
where $A^\e \defi \{x\in \R: \text{d}(A,x) \leq \e\}$ is the $\e$-neighborhood of the set $A$. It is a classical fact that the Prohorov distance metricizes weak convergence. We also recall the following implication, as it will be used at different occurences: for two r.v.s $X,Y$, slightly abusing notation, one has:
\beq \label{gabriel}
\PP(|X-Y|>\varepsilon)\leq \varepsilon \Rightarrow \text{d}_{\text{P}}(X,Y) \leq \varepsilon\,.
\eeq
In fact, $ \PP(|X-Y|>\varepsilon)\leq \varepsilon$ implies that for $A\subset \R$,
\beq \PP( X\in A) \leq \PP(X\in A, |X-Y|\leq \varepsilon) + \PP(|X-Y|>\varepsilon) \leq \PP(Y \in A^{\varepsilon}) + \varepsilon\,
\eeq
from which $\text{d}_{\text{P}}(X,Y) \leq \varepsilon$ follows, settling \eqref{gabriel}.\\

We now proceed to the induction step:  we thus assume that $Z_{r,n}$ converges weakly to $Z_{r}$ for some $r\in \N$ and show how to deduce that $Z_{r+1, n}$ converges weakly to  $Z_{r+1}$. First, we observe that by definition
\beq \bea \label{small_sum}
 Z_{r+1,n} & = \sum\limits_{i\leq n} \exp\left(-n \xi_{(\boldsymbol{0},e_i)}\right) \sum\limits_{x \in \mathcal{V}^{\rightarrow}_{r+1,n}: x_2 = e_i }\exp\left( -n \sum\limits_{l=2}^{r+1} \xi_{[x]_l}\right) \\
& = \sum\limits_{i\leq n} \exp\left(-n \xi_{(\boldsymbol{0},e_i)}\right) \times Z_{r, n}^{e_i}\,,
\eea \eeq
changing notation for the second sum to lighten exposition.

We claim that it suffices to consider small $\xi$-values in the first sum. Precisely, let $\varepsilon>0$, set $K_\varepsilon = -2\ln\varepsilon $, and restrict the first sum to those $\xi's$ such that $\xi_{(0, e_i)} \leq K_\vare/n$. We claim that this causes only an $\vare$-error in Prohorov distance, to wit
\beq  \label{dist_p}
\sup_{n,r} \text{d}_\text{P} \left( Z_{r+1, n}, \sum\limits_{i\leq n}\mathbbm{1}_{\{\xi_{(\boldsymbol{0},e_i)} \leq K_\varepsilon/n\}} e^{-n \xi_{(\boldsymbol{0},e_i)}} \times Z_{r, n}^{e_i} \right) \leq \vare.
\eeq
In fact,  for the contribution of large $\xi's$, it holds:
\beq \bea \label{error_p}
& \PP\left(\sum\limits_{i\leq n}\mathbbm{1}_{\{\xi_{(\boldsymbol{0},e_i)}>K_\varepsilon/n\}} e^{-n \xi_{(\boldsymbol{0},e_i)}} \times Z_{r, n}^{e_i} > \varepsilon\right) \\
& \hspace{3cm}  \leq \frac{1}{\varepsilon}\E\left[ \sum\limits_{i\leq n}\mathbbm{1}_{\{\xi_{(\boldsymbol{0},e_i)}> K_\varepsilon/n\}} e^{-n \xi_{(\boldsymbol{0},e_i)}}  \times Z_{r, n}^{e_i} \right] \\
& \hspace{3cm}  = \frac{n}{\vare} \E\left[ \mathbbm{1}_{\{\xi_{(\boldsymbol{0},e_i)}> K_\varepsilon/n\}} e^{-n \xi_{(\boldsymbol{0},e_i)}}  \right] \times \E\left[ Z_{r, n}^{e_i} \right] \,,
\eea \eeq
the first step by Markov inequality, and the second by independence.  Computing explicitly the above expectations yields that the r.h.s. of \eqref{error_p} is {\it at most}
\beq \bea \label{error_p_2}
\frac{n}{\varepsilon}   \int\limits_{K_\varepsilon/n}^{\infty} e^{-(n+1)x} dx \times   \frac{(n-1)!}{(n-r-1)!} (n+1)^{-r}
& \leq \frac{\exp -K_\varepsilon}{\varepsilon} = \varepsilon,
\eea \eeq
since $K_\vare= -2\ln\varepsilon$. This settles \eqref{dist_p}. 

Consider now the permutation $p$ of $\{1,..,n\}$ such that $(\xi_{p(i)})_{i\leq n}$ is increasing, and set $\hat{K}_\varepsilon \defi \lceil K_\varepsilon/\varepsilon \rceil$. We clearly have 
\beq
Z_{n, r+1} \geq \sum\limits_{i\leq \hat{K}_\varepsilon} e^{-n \xi_{p(i)}} Z_{r, n}^{e_{p(i)}} \,.
\eeq
for $\geq\hat{K}_\varepsilon$. On the other hand,
\beq\bea
& \PP\left( Z_{n, r+1} \geq  \sum\limits_{i\leq \hat{K}_\varepsilon} e^{-n \xi_{p(i)}} Z_{r, n}^{e_{p(i)}} + \varepsilon \right) \\
&  \hspace{1.5cm} \leq \PP\left( Z_{n, r+1} \geq  \varepsilon+ \sum\limits_{i\leq n}\mathbbm{1}_{\{\xi_{(\boldsymbol{0},e_i)} \leq K_\varepsilon/n\}} e^{-n \xi_{(\boldsymbol{0},e_i)}} \times Z_{r, n}^{e_i} \right)\\
& \hspace{3cm} + \PP\left( \sum\limits_{i\leq \hat{K}_\varepsilon} e^{-n \xi_{p(i)}} Z_{r, n}^{e_{p(i)}} \leq \sum\limits_{i\leq n}\mathbbm{1}_{\{\xi_{(\boldsymbol{0},e_i)} \leq K_\varepsilon/n\}} e^{-n \xi_{(\boldsymbol{0},e_i)}} \times Z_{r, n}^{e_i} \right)
\eea\eeq
While the first term is at most $\varepsilon$ by \eqref{error_p} and \eqref{error_p_2}, the second term equals
\beq \bea
& \PP\left( \#\{i\leq n: \xi_{(\boldsymbol{0},e_i)} \leq K_\varepsilon/n \} > \hat{K}_\varepsilon \right) \\
& \hspace{3cm} \leq n \PP(\xi_{(\boldsymbol{0},e_1)}\leq  K_\varepsilon/n) \big/ \hat{K}_\varepsilon  \leq K_\varepsilon \big/ \hat{K}_\varepsilon \leq \varepsilon\,,
\eea \eeq
the first estimate by Markov inequality and the second using $(1-e^{-x}) \leq x$.

All in all, in virtue of \eqref{gabriel}, the above considerations imply that
\beq \label{eq_n_lim1}
\sup_{n,r}\text{d}_\text{P}\left( Z_{n, r+1}, \sum\limits_{i\leq \hat{K}_\varepsilon} e^{-n \xi_{p(i)}} Z_{r, n}^{e_p(i)} \right) \leq 2 \vare \,.
\eeq
A fixed, finite number of paths therefore carries essentially all weight: we will now show that these paths are, with overwhelming probability, organised in a "tree-like fashion".  Towards this goal,  we go back to the original formulation
\beq \bea \label{eq_n_lim2}
& \sum\limits_{i\leq \hat{K}_\varepsilon} e^{-n \xi_{p(i)}} Z_{r,n}^{e_i} = \sum\limits_{i\leq \hat{K}_\varepsilon} e^{-n \xi_{p(i)}}  \sum\limits_{x \in \mathcal{V}^{\rightarrow}_{r+1,n}: x_2 = e_{p(i)} }\exp\left( -n \sum\limits_{l=2}^{r+1} \xi_{[x]_l}\right)\,.
\eea \eeq
Note that any directed path of length $r+1$ with first step $(\boldsymbol{0},e_i)$, can only share an edge with another path starting with $(\boldsymbol{0},e_j)$, $i\neq j$ if it goes in the direction $e_j$ at some point. By this observation for $i\neq j$ and $i,j \in \{1,..,n\}$
\beq
|\{x \in \mathcal{V}^{\rightarrow}_{r+1,n}: x_2 = e_i,  \exists x'\in  \mathcal{V}^{\rightarrow}_{r+1,n} \mbox{ s.t. } x'_2 = e_j \mbox{ and } x\cap x' \neq \emptyset \} | = O(n^{r-1})
\eeq
holds. Combining this fact with the observation
\beq
\E\exp\left( -n \sum\limits_{l=2}^{r+1} \xi_{[x]_l}\right) = (n+1)^{-r}
\eeq
we see that the total contribution of such paths  converges in probability to zero, by Markov inequality, and swapping these intersecting summands for copies of themselves that are independent of paths with different start edge does not change the weak limit. The weak limit of $\eqref{eq_n_lim2}$ therefore coincides with the weak limit of
\beq
\sum\limits_{i\leq \hat{K}_\varepsilon} \exp\left(-n \xi_{p(i)}\right) \sum\limits_{x \in \mathcal{V}^{\rightarrow}_{r,n-1} }\exp\left( -n \sum\limits_{l=1}^{r} \xi^{(p(i))}_{[x]_l}\right)
\eeq
where $\xi^{(p(i))}_{[x]_l} = \xi_{[x]_l}$ if $[x]_l$ cannot be part of a path starting with $e_{p(j)}$ for some $j\neq i$ with $j\leq  \hat{K}_\varepsilon$. On the other hand, the $\xi^{(p(i))}_{[x]_l}$'s are exponentially distributed and independent of each other for different $p(i)$ and or different $[x]_l$ as well as independent of all $(\xi_{e})_{e\in\E_n}$.  Finally, we realize that replacing
\beq
\exp\left( -n \sum\limits_{l=1}^{r} \xi^{(p(i))}_{[x]_l}\right) \mbox{\quad by \quad}\exp\left( -(n-1) \sum\limits_{l=1}^{r} \xi^{(p(i))}_{[x]_l}\right)
\eeq
causes, by the restriction argument \eqref{cond_exp_eq2}, an error which vanishes in probability. Collecting all changes and estimates, we have thus shown that the distribution of $Z_{r+1,n}$ is at most $2\varepsilon+o_n(1)$-Prohorov distance away from the weak limit of
\beq \label{eq_n_lim2}
\sum\limits_{i\leq \hat{K}_\varepsilon} \exp\left(-n \xi_{p(i)}\right) Z^{(i)}_{r,n-1},
\eeq
where ${Z}_{r,n-1}^{(i)},i\in \N$ are independent copies of ${Z}_{r,n-1}$. By assumption $Z_{r,n-1}$ converges weakly to $Z_r$ and by Lemma $\ref{init}$ the smallest finitely many $n\xi$'s converge weakly to the first that many points of a PPP($\mathbbm{1}_{\R^+} dx$). We conclude that the Prohorov distance of $Z_{r+1,n}$ and
\beq
\sum\limits_{i\leq \hat{K}_\varepsilon} \exp\left(- \hat{\eta_i} \right) Z^{(i)}_{r},
\eeq
is at most by an in $n$ vanishing sequence larger than $2\varepsilon$.  Checking using Markov inequality that the contibution of $i>\hat{K}_\varepsilon$ is vanishing in probability gives that
\beq
d_{P}\left(\mathcal{L} (Z_{r+1,n}), (\mathcal{L} Z_{r+1}) \right) \rightarrow 0
\eeq
has to hold as $n\rightarrow \infty$. This finishes the induction, and the proof of the n-convergence is thus settled. \\
${}$  \hfill $\square$

We move to the proof of the second claim of Proposition \ref{double_limit}, the r-convergence. As mentioned,
this will be done via a contraction argument on the space $\mathcal{P}_2$ of probability measures on $\R$ with  finite second moment. To this end, let $(\eta_i)_{i\in\N}$ be a PPP($\mathbbm{1}_{\R^+} dx$). Define
\beq \bea
T: \mathcal{P}_2 & \to \mathcal{P}_2, \\
 \mu & \mapsto \mathcal{L}\left(\sum\limits_{i\in\N} e^{-\eta_i} X_i\right),
\eea \eeq
where $(X_i)_{i\in \N}$ are independent and identically $\mu$-distributed, and independent of $\eta$.
Note that $T$ is well-defined, i.e., we have that $T\mu$ has a finite second moment for all $\mu\in \mathcal{P}_2$ by applying the triangle inequality, $\E[\sum\limits_{i\in\N}e^{-2\eta_i}]=1/2$ and independence. Moreover, since $\E[\sum_{i\in\N} e^{-\eta_i}]=1$ the map $T$ does not change the first moment. Hence, for the subset
$$\mathcal{P}_{2,1}:=\left\{\mu \in \mathcal{P}_{2}: \int x\; d\mu = 1\right\}$$ the restriction of $T$ to $\mathcal{P}_{2,1}$ maps to $\mathcal{P}_{2,1}$.  By construction, it holds that
\beq \label{fix}
\mathcal{L} (Z_{r+1}) = T\mathcal{L}(Z_r).
\eeq
We now endow $\mathcal{P}_2$ with the minimal $L_2$-distance $\ell_2$, also called Wasserstein distance of order $2$: for $\mu,\nu\in \mathcal{P}_{2}$ this is defined by
 $$\ell_2(\mu,\nu)=\inf\{\|V-W\|_2:\mathcal{L}(V)=\mu,\mathcal{L}(W)=\nu\},$$
 where the infimum is over all random variables $V,W$ on a joint probability space with the respective distributions. Convergence in $\ell_2$ implies weak convergence, $(\mathcal{P}_{2},\ell_2)$ and  $(\mathcal{P}_{2,1},\ell_2)$ are complete metric spaces. For these topological properties and the existence of optimal couplings used below see, e.g., Ambrosio, Gigli and Savar{\'e} \cite{ags} or Villani \cite{villani}. Within the present setting, in order to prove the r-convergence it suffices to prove that
\begin{itemize}
\item The restriction of $T$ to $\mathcal{P}_{2,1}$ is a strict $\ell_2$-contraction.
\item The standard exponential distribution is a fixed point of $T$ restricted to $\mathcal{P}_{2,1}$.
\end{itemize}
We remark that $T$ as a map on $\mathcal{P}_{2}$ has infinitely many fixed points and that our argument below also implies that these fixed points are exactly the exponential distributions with arbitrary parameter, their negatives, and the Dirac measure in $0$. Uniqueness of the fixed point on  $\mathcal{P}_{2,1}$ is immediate by Banach fixed point theorem and the strict contraction property. 

Contractivity goes as follows. For $\mu,\nu\in \mathcal{P}_{2,1}$, let $(X_i,Y_i)_{i\in\N}$ be a sequence of independent {\it optimal} $\ell_2$-couplings, which are also independent of $\eta$; optimal $\ell_2$-couplings means here that the pair $(X_i,Y_i)$ has marginal distributions $\mu$ and $\nu$, and that it attains the infimum  in the definition of $\ell_2$. It then holds:
\beq
\ell_2(T\mu,T\nu)^2 \leq \E\left[ \left( \sum\limits_{i\in\N} e^{-\eta_i} (X_i-Y_i)\right)^2\right].
\eeq
Remark that the off-diagonal terms on the r.h.s. above vanish, since $X_i-Y_i$ has zero expectation: using this, we thus obtain
\beq
\ell_2(T\mu,T\nu)^2 \leq  \E\left[ \sum_{i} e^{-2\eta_i}\right]\E\left[ (X_1-Y_1)^2\right] =\frac{1}{2} \ell_2\left(\mu,\nu\right)^2,
\eeq
the last step by optimality of the coupling. This implies that the restriction of the map $T$ to $\mathcal{P}_{2,1}$ is an $\ell_2$-contraction. 

It thus remains to prove that the standard exponential distribution is the fixed point of $T$ in $\mathcal{P}_{2,1}$. This can be checked via  Laplace transformation: consider independent standard exponentials $X_1, X_2,...$ which are also independent of $\eta$. For $t> 0$,
\beq \bea
\E\left[\exp\left(-t \sum\limits_{i=1}^\infty e^{-\eta_i} X_i\right)\right] & = \E\left[ \exp\left(-\sum\limits_{i=1}^\infty \ln\left( 1+t e^{-\eta_i} \right)\right)\right] \\
& = \exp\left(\int\limits_{0}^{\infty} \frac{1}{1+t e^{-x}} -1 dx\right) = \frac{1}{1+t},
\eea \eeq
which is the Laplace transform of a standard exponential. This implies {\it ii)}. The r-convergence therefore
immediately follows from Banach fixed point theorem. \\
${}$ \hfill $\square$

\subsection{Vanishing of the Chen-Stein term.} \label{vanishing_cs} The goal here is to prove \eqref{cs_vanishes},
namely that
\beq
\lim_{r\to \infty} \lim_{n\to \infty} \; \text{CS}(r,n) = 0\,.
\eeq
This requires some additional notation. Let
\[ \bea
\Sigma_{n,r} \defi & \Big\{ (\pi,\pi') \in \Sigma_n\times \Sigma_n: \pi, \pi' \; \text{have at least a common edge}\;  e, \\
& \hspace{3cm} e = (u,v) \in E, \{d(u,\boldsymbol{0}),d(v,\boldsymbol{0})\}\in \left[r, n-r\right) \Big\}\,.
\eea \]
For paths $(\pi,\pi') \in \Sigma_n\times \Sigma_n$, we denote by $\pi\wedge \pi'$ their {\it overlap}, i.e. the number of edges shared by both paths. Working out the conditional Chen-Stein bound \eqref{ccs}, we get
\beq \bea \label{letsgo}
\text{CS}(r,n) = & \E d_{TV,\mathcal{F}_{r,n}} \left( \Xi_n(A), \Poi\left(\la_n(A) \right) \right)\\
& \leq \; \sum_{\pi\in \Sigma_n} \E\left[ \E[I_\pi(A)|\mathcal{F}_{r,n}]^2 \right] \\
& \qquad \qquad + \sum_\star \E \left[ \E[I_\pi(A)|\mathcal{F}_{r,n}] \E[I_{\pi'}(A)|\mathcal{F}_{r,n}] \right] \\
& \qquad + \sum_\star \E\left[ \E[I_\pi(A) I_{\pi'}(A)|\mathcal{F}_{r,n}] \right]\,,
\eea \eeq
where $\sum_\star$ denotes summation over all $(\pi,\pi')\in \Sigma_{n,r} : 1\leq\pi\wedge \pi'\leq n-2$. We will prove that all three terms on the r.h.s. of \eqref{letsgo} vanish in the limit $n\to \infty$ first, and $r\to \infty$ next. As the proof is long and technical, we formulate the statements in the form of three Lemmata.

\begin{lem} \label{uno}
\[
\lim_{r\to \infty }\lim_{n \to \infty} \sum_{\pi\in \Sigma_n} \E\left[ \E[I_\pi(A)|\mathcal{F}_{r,n}]^2 \right] = 0\,.
\]
\end{lem}

\begin{lem} \label{due}
\[
\lim_{r\to \infty }\lim_{n \to \infty} \sum_\star \E \left[ \E[I_\pi(A)|\mathcal{F}_{r,n}] \E[I_{\pi'}(A)|\mathcal{F}_{r,n}] \right] = 0\,.
\]
\end{lem}

\begin{lem} \label{tre}
\[
\lim_{r\to \infty }\lim_{n \to \infty}  \sum_\star \E\left[ \E[I_\pi(A) I_{\pi'}(A)|\mathcal{F}_{r,n}] \right] = 0\,.
\]
\end{lem}

The first contribution is easily taken care of:

\begin{proof}[Proof of Lemma \ref{uno}]
By symmetry we have that
\beq\bea \label{letsgo_1}
\sum_{\pi\in \Sigma_n}\E[\E[I_\pi(A)|\mathcal{F}_{r,n}]^2] = n!\E[\E[I_{\pi^{*}}(A)|\mathcal{F}_{r,n}]^2]\,,
\eea\eeq
where $\pi^{*} \in \Sigma_n$ is arbitrary. It thus follows from the tail-estimates of Lemma \ref{tail}  that
\beq\bea \label{letsgo_2}
\eqref{letsgo_1} & =n!\int_{0}^{1+\frac{a}{n}} {\left(1+K(1+\frac{a}{n}-x,n-2r) \right)}^2\frac{e^{-2(1+\frac{a}{n})+x}{(1+\frac{a}{n}-x)}^{2n-4r}x^{2r-1}}{{(n-2r)!}^2(2r-1)!}dx\\
& \leq n!\int_{0}^{1+\frac{a}{n}} {\left(1+e^{(1+\frac{a}{n})}\frac{(1+\frac{a}{n})}{n-2r} \right)}^2\frac{e^{-(1+\frac{a}{n})}{(1+\frac{a}{n})}^{2n-4r}{(1+\frac{a}{n})}^{2r-1}}{{(n-2r)!}^2(2r-1)!}dx\\
& =   \frac{n!e^{2a}}{{(n-2r)!}^2 (2r-1)!} (1+o_n(1))\,. \\
\eea\eeq
Since the r.h.s. of \eqref{letsgo_2} is vanishing in the large $n$-limit, the proof of Lemma \ref{uno} is concluded.
\end{proof}

Lemma \ref{due} and \ref{tre} require more work. In particular, we will make heavy use of the following combinatorial estimates, which have been established by Fill and Pemantle \cite{Fill_Pemantle} (see Lemma 2.3, 2.4 and 2.5 p. 598):

\begin{prop}[Path counting] \label{path_counting} Let $\pi'$ be any reference path on the $n$-dim hypercube connecting $\boldsymbol{0}$ and $\boldsymbol 1$. Denote by $f(n, k)$ the number of paths $\pi$ that share precisely $k$ edges ($k\geq1$) with $\pi'$. Finally, shorten $\mathfrak{n_{e}} \defi n-5e(n+3)^{2/3}$.
\begin{itemize}
\item For any  $K(n) = o(n)$ as $n \to \infty$,
\beq \label{path_counting_i}
f(n,k) \leq (1+o(1))(k+1)(n-k)! \,
\eeq
uniformly in $k$ for $k\leq K(n).$
\item Suppose $k\leq \mathfrak{n_{e}}$. Then, for $n$ large enough,
\beq \label{path_counting_ii}
f(n,k)\leq n^6(n-k)! \,.
\eeq
\item Suppose $k \geq \mathfrak{n_{e}}$. Then, for $n$ large enough,
\beq \label{path_counting_iii}
f(n,k)\leq {(2n^{\frac{7}{8}})}^{n-k} (n-k+1) \,.
\eeq
\end{itemize}
\end{prop}

\begin{proof}[Proof of Lemma \ref{due}]
Here and below, $\kappa_a>0$ will denote a universal constant not necessarily the same at different occurences, and which depends solely on $a$. By symmetry,
\beq
\sum_\star \E[I_\pi(A) I_{\pi'}(A)] = n!\sum_{\star, \star} \E[I_{\pi^{*}}(A) I_{\pi'}(A)]
\eeq
where $\pi^{*} \in \Sigma_n$ is arbitrary and $\sum_{\star, \star}$ standing for summation over
$$\pi'\in \Sigma_n: (\pi^{*},\pi')\in \Sigma_{n,r}  , 1 \leq \pi^{*}\wedge \pi' \leq n-2.$$
Let $k \in \{1,n-2\}$ and $\pi'\in \Sigma_n, \pi^{*}\wedge \pi' =k$. Splitting $X_{\pi^{*}}$ and $X_{\pi'}$ into common/non-common edges, we obtain
\beq\bea
\E[I_{\pi^{*}}(A) I_{\pi'}(A)]&=\PP\left( X_{\pi^{*}} \leq 1+\frac{a}{n}, X_{\pi'} \leq 1+\frac{a}{n} \right)\\
&=\int_{\R}\PP\left(x+X_{n-k}\leq 1+\frac{a}{n},x+X'_{n-k}\leq 1+\frac{a}{n}\mid X_k=x\right)\PP( X_k \in dx)\,.
\eea\eeq
In the above, $X_{n-k}$ and $X'_{n-k}$ correspond to the compound weights of the non-common edges: these are Gamma$(n-k,1)$-distributed random variables; $X_{k}$ corresponds to the weight of the common edges: this is a Gamma$(k,1)$-distributed random variable. By construction, $X_{n-k}, X'_{n-k}$ and $X_{k}$ are independent. All in all,
\beq\bea \label{blaaa}
\E[I_{\pi^{*}}(A) I_{\pi'}(A)]&= \int_{0}^{+\infty}\PP\left(x+X_{n-k}\leq 1+\frac{a}{n}\right)^2\frac{e^{-x}x^{k-1}}{(k-1)!}dx\\
&\leq \frac{\kappa_a}{{(n-k)!}^2}\int_{0}^{1+\frac{a}{n}}{\left(1+\frac{a}{n}-x\right)}^{2(n-k)}\frac{x^{k-1}}{(k-1)!}dx\,.
\eea\eeq
The last inequality by the tail-estimate of Lemma $\ref{tail}$. Integration by parts then yields
\beq\bea
\int_{0}^{1+\frac{a}{n}}{\left(1+\frac{a}{n}-x\right)}^{2(n-k)}x^{k-1}dx\leq \kappa_a \frac{(k-1)!(2(n-k))!}{(2n-k)!}\,.
\eea\eeq
and therefore
\beq\bea \label{est1}
E[I_{\pi^{*}}(A) I_{\pi'}(A)]&\leq \kappa_a \frac{(2(n-k))!}{(2n-k)!{(n-k)!}^2}.
\eea\eeq
Denoting by $f(n, k, r)$ the number of paths $\pi'$ that share precisely $k$ edges ($1 \leq k\leq n-2$) with $\pi^{*}$ and that satisfy $(\pi',\pi^{*})\in \Sigma_{n,r}$, we thus have that
\beq\bea \label{est2}
n!\sum_{\star, \star} \E[I_{\pi^{*}}(A) I_{\pi'}(A)] &=n!\sum_{k=1}^{n-2} f(n, k, r) \E[I_{\pi^{*}}(A) I_{\pi'}(A)] \\
& \stackrel{\eqref{est1}}{\leq}\kappa_a \sum_{k=1}^{n-2} \frac{f(n, k, r)}{(n-k)!}\times \frac{n!(2(n-k))!}{(n-k)!(2n-k)!} \\
& \; \leq \kappa_a \sum_{k=1}^{n-2} \frac{f(n, k, r)}{(n-k)!} \times \frac{(1-\frac{k}{n})^{n-k}}{2^k(1-\frac{k}{2n})^{2n-k}},
\eea\eeq
the last inequality by Stirling approximation. To lighten notation, remark that with $\gamma \defi k/n \in [0,1]$, the second factor in the last sum above can be written as
\beq \bea \label{g-fct}
\frac{(1-\frac{k}{n})^{n-k}}{2^k(1-\frac{k}{2n})^{2n-k}}={\left(\frac{{(4(1-\gamma))}^{(1-\gamma)}}{{(2-\gamma)}^{(2-\gamma)}}\right)}^n\defi g(\gamma)^n\,.
\eea \eeq
With this, \eqref{est2} takes the form
\beq\bea \label{estimate}
n!\sum_{\star, \star} \E[I_{\pi^{*}}(A) I_{\pi'}(A)] \leq \kappa_a  \sum_{k=1}^{n-2} \frac{f(n, k, r)}{(n-k)!}\times {g\left(\frac{k}{n}\right)}^n.
\eea\eeq
The following observation, whose elementary proof is postponed to the end of this section, will be useful.
\begin{fact} \label{really_technical} The function $g: [0,1] \to \R_+$ defined \eqref{g-fct}
is increasing on $[2/3, 1)$. Furthermore,
\beq \label{34}
\forall \gamma \leq 2/3: \; g(\gamma) \leq \left(\frac{3}{4}\right)^{\gamma}\,.
\eeq
\end{fact}
\noindent In view of Proposition \ref{path_counting}, recalling that $\mathfrak{n_{e}} = n-5e(n+3)^{2/3}$ and with
\beq \label{def_c}
C \defi \frac{7}{ \ln\left( 4/3 \right)}\,,
\eeq
 we split the sum on the r.h.s. of \eqref{estimate} into three regimes, to wit:

\beq\bea \label{three sums}
\left( \sum_{k=1}^{C\ln(n)}+\sum_{k=C \ln(n)+1}^{\mathfrak{n_{e}}}+\sum_{k=\mathfrak{n_{e}}+1}^{n-2} \right) \frac{f(n, k, r)}{(n-k)!}  \times g\left(\frac{k}{n}\right)^n \,.
\eea\eeq
Concerning the first sum :
\beq\bea
\sum_{k=1}^{C \ln(n)} \frac{f(n, k, r)}{(n-k)!}{g\left(\frac{k}{n}\right)}^n &\stackrel{\eqref{34}}{\leq} \sum_{k=1}^{C \ln(n)} \frac{f(n, k, r)}{(n-k)!}\left(\frac{3}{4}\right)^k\\
&\leq \sum_{k=1}^{r-1} \frac{f(n, k, r)}{(n-r+1)!}\left(\frac{3}{4}\right)^k +\sum_{k=r}^{C \ln(n)} \frac{f(n, k)}{(n-k)!}\left(\frac{3}{4}\right)^k\\
& \leq \sum_{k=1}^{r-1} \frac{f(n, k, r)}{(n-r+1)!}\left(\frac{3}{4}\right)^k +\kappa_ a\sum_{k=r}^{C \ln(n)}(k+1)\left(\frac{3}{4}\right)^k\,,
\eea\eeq
by Proposition \ref{path_counting}.

The function $f(n, k, r)$ counts the number of paths $\pi'$ that share precisely $k$ edges ($1 \leq k\leq n-2$) with $\pi^{*}$ and that satisfy $(\pi',\pi^{*})\in \Sigma_{n,r}$: we claim that
\beq \label{claiming}
f(n, k, r)\leq r!(n-r-1)!n.
\eeq
To see this, recall that the vertices of the hypercube stand in correspondence with the standard basis of $\R^n$: every edge is parallel to some unit vector $e_j$, where $e_j$ connects $(0, \dots, 0)$ to $(0, \dots, 0, 1, 0, \dots, 0)$
with a $1$ in position $j$. We identify a directed path $\pi$ from $\boldsymbol 0$ to $\boldsymbol 1$ by a permutation of $1 2 \dots n$, say $\pi_1 \pi_2\dots \pi_n$. $\pi_l$ is giving the direction the path $\pi$ goes in step $l$, hence after $i$ steps the path $\pi_1 \pi_2\dots \pi_n$ is at vertex $\sum_{j\leq i} e_{\pi_j}$. (By a slight abuse of notation, $\pi_1$ will refer here below to a number between, $1$ and $n$). Let  now $\pi^{*}$ be the reference path, say $\pi^{*}=12...n$. We set $u_i = l$ if the $l$-th traversed edge by $\pi'$ is the $i$-th shared edge of $\pi'$ and $\pi^{*}$, setting by convention $r_0 =0$ and $r_{k+1} = n+1$. Shorten then $\textbf{u}\defi \textbf{u}(\pi')=(u_0,...,u_{k+1})$, and  $s_i \defi u_{i+1}-u_{i}$, $i=0,...,k$. For any sequence $\textbf{u}_0=(u_0,...,u_{k+1})$ with $0=u_0<u_1<...<u_k<u_{k+1}=n+1$, let $C(\textbf{u}_0)$ denote the number of paths $\pi'$ with $\textbf{u}(\pi')=\textbf{u}_0$. Since the values $\pi'_{u_i+1},...,\pi'_{u_i+s_i-1}$ must be a permutation of $\{u_i+1,...,u_i+s_i-1\}$, one easily sees that $C(\textbf{u})\leq G(\textbf{u})$, where
\beq\label{G}
G(\textbf{u})=\prod\limits_{i=0}^{k}(s_i-1)! \,.
\eeq
We also observe that two such paths must have a common edge in the {\it middle region} $(\pi',\pi^{*})\in \Sigma_{n,r}$. Let $e$ be such an edge: as it turns out, this is quite restrictive. Indeed, it implies that there exists $u_j \in \{r+1,n-r\}$ for $ j \in \{1,..., k\}$.
In virtue of \eqref{G} and log-convexity of factorials, one has at most $r!(n-r-1)!$ paths $\pi'$ sharing the edge $e$ with the reference-path $\pi^{*}$, and at most $\binom {n}{1} = n$ ways to choose this edge: combining all this settles \eqref{claiming}.

It follows that
\beq\bea
\sum_{k=1}^{C\ln(n)} \frac{f(n, k, r)}{(n-k)!}{g\left(\frac{k}{n}\right)}^n &\leq \sum_{k=1}^{r-1} \frac{r!(n-r-1)!n}{(n-r+1)!}\left(\frac{3}{4}\right)^k +\kappa_a\sum_{k=r}^{+\infty} (k+1)\left(\frac{3}{4}\right)^k.
\eea\eeq
The first sum above clearly tends to $0$ as $n \rightarrow \infty$, whereas the second sum vanishes when $r \rightarrow \infty$: the first regime in \eqref{three sums} therefore yields no contribution in the double limit. \\

\noindent As for the second regime, by Proposition \ref{path_counting},
\beq\bea\label{reg2}
\sum_{k=C \ln(n)}^{\mathfrak{n_{e}}} \frac{f(n, k, r)}{(n-k)!}{g\left(\frac{k}{n}\right)}^n &\leq \sum_{k=C \ln(n)}^{\mathfrak{n_{e}}} \frac{f(n, k)}{(n-k)!}{g\left(\frac{k}{n}\right)}^n\\
&\leq n^6 \sum_{k=C \ln(n)}^{\mathfrak{n_{e}}}{g\left(\frac{k}{n}\right)}^n\\
&=n^6 \left(\sum_{k=C \ln(n)}^{2n/3} {g\left(\frac{k}{n}\right)}^n+\sum_{k=2n/3+1}^{\mathfrak{n_{e}}} {g\left(\frac{k}{n}\right)}^n\right).
\eea\eeq
As pointed out in Fact \ref{really_technical}, the $g$-function is increasing on $[2/3,1)$, whereas on the "complement" \eqref{34} holds: these observations, together with \eqref{reg2} imply that
\beq\bea \label{neverending}
\sum_{k=C \ln(n)}^{\mathfrak{n_{e}}} \frac{f(n, k, r)}{(n-k)!}{g\left(\frac{k}{n}\right)}^n &\leq n^6 \left(\sum_{k=C \ln(n)}^{2n/3} \left(\frac{3}{4}\right)^k+\sum_{k=2n/3+1}^{\mathfrak{n_{e}}} {g\left(\frac{\mathfrak{n_{e}}}{n} \right)}^n\right)\\
&\leq 4n^6 \left(\frac{3}{4}\right)^{C \ln(n)}+ n^7 {g\left(\frac{\mathfrak{n_{e}}}{n}\right)}^n\\
& = 4 \exp\left\{ (6+ C \ln(3/4)) \ln(n) \right\} + n^7 {g\left(\frac{\mathfrak{n_{e}}}{n}\right)}^n\,.\\
\eea \eeq
In virtue of the choice \eqref{def_c} we have that $6+ C \ln(3/4) = -1$, hence
\beq \label{neverending_2}
\eqref{neverending} = o_n(1)+ n^7 {g\left(\frac{\mathfrak{n_{e}}}{n}\right)}^n\,.
\eeq
By definition of the $g$-function \eqref{g-fct} and $\mathfrak{n_e}$, it holds:
\beq \bea \label{estg}
{g\left(\frac{\mathfrak{n_{e}}}{n}\right)}^n& = \frac{(1-\frac{\mathfrak{n_{e}}}{n})^{n-\mathfrak{n_{e}}}}{2^{\mathfrak{n_{e}}}(1-\frac{\mathfrak{n_{e}}}{2n})^{2n-\mathfrak{n_{e}}}} \\
& =  \left(\frac{5e{(n+3)}^{\frac{2}{3}}}{n}\right)^{5e{(n+3)}^{\frac{2}{3}}}2^{10e{(n+3)}^{\frac{2}{3}}} {\left(1+\frac{5e{(n+3)}^{\frac{2}{3}}}{n}\right)}^{-n-5e{(n+3)}^{2/3} }.\\
\eea \eeq
Notice that
\beq
1+\frac{5e{(n+3)}^{\frac{2}{3}}}{n}\geq 1 \text{ and } {(n+3)}^{\frac{2}{3}} \leq 2n^{\frac{2}{3}} \text{ for } n\geq 3,
\eeq
thus
\beq
\eqref{estg} \leq \left(\frac{40e}{n^{1/3}}\right)^{10e{n}^{\frac{2}{3}}}=o(n^{-7})\,,
\eeq
implying that the second regime in \eqref{three sums} yields no contribution in the limit $n \to +\infty$. \\

\noindent As for the third, and last regime: by definition of the $g$-function,
\beq\bea \label{never}
\sum_{k=\mathfrak{n_{e}}+1}^{n-2} \frac{f(n, k, r)}{(n-k)!}{g\left(\frac{k}{n}\right)}^n &\leq \sum_{k=n_e+1}^{n-2} \frac{f(n, k)}{(n-k)!}\frac{(1-\frac{k}{n})^{n-k}}{2^k(1-\frac{k}{2n})^{2n-k}}\\
& \leq \sum_{k=\mathfrak{n_{e}}+1}^{n-2} \frac{{(2n^{\frac{7}{8}})}^{n-k} (n-k+1)}{(n-k)!}\frac{(1-\frac{k}{n})^{n-k}}{2^k(1-\frac{k}{2n})^{2n-k}}\,,
\eea \eeq
the last step in virtue of Proposition \ref{path_counting}. By change of variable, $n-k \mapsto u$,  we get
\beq \bea
\eqref{never}  & = \sum_{u=2}^{5e(n+3)^{\frac{2}{3}}-1} {\left(\frac{8un^{\frac{7}{8}}}{n}\right)}^u\frac{(u+1)}{{(1+\frac{u}{n})}^{n+u}u!} \leq \sum_{u=2}^{\infty} {\left(\frac{8e}{n^{\frac{1}{8}}}\right)}^u(u+1)\,,
\eea\eeq
the last inequality by Stirling's approximation. It thus follows that the contribution of the third and last regime in \eqref{three sums} also vanishes as $n \to +\infty$. The proof of Lemma \ref{due} is concluded.
\end{proof}

We finally provide the elementary

\begin{proof}[Proof of Fact \ref{really_technical}] The sign of $g'$ is given by the sign of
\[\frac{d}{d\gamma}\left(\ln(4-4\gamma)(1-\gamma)-\ln(2-\gamma)(2-\gamma)\right)=\ln\left(\frac{2-\gamma}{4-4\gamma}\right).\]
It follows that $g'(\gamma)\leq 0$ $\forall \gamma\leq 2/3$ and $g'(\gamma)\geq 0$ $\forall \gamma\geq 2/3$. Furthermore, since
\[
1-\gamma \leq {\left(1-\frac{\gamma}{2} \right)}^2\,,
\]
we have
\beq\bea \label{estimate2}
g(\gamma) = \frac{{(4(1-\gamma))}^{(1-\gamma)}}{{(2-\gamma)}^{(2-\gamma)}}
\leq {(2-\gamma)}^{-\gamma}
\leq \left(\frac{3}{4}\right)^\gamma\,,
\eea\eeq
$\forall \gamma\leq 2/3$, settling \eqref{34}.
\end{proof}

\begin{proof}[Proof of Lemma \ref{tre}]
Again by symmetry,
\beq\bea \label{estimate3}
&\sum_\star \E[\E[I_\pi(A)|\mathcal{F}_{r,n}] \E[I_{\pi'}(A)|\mathcal{F}_{r,n}]] \\
& \qquad =n!\sum_{\star, \star} \E[\E[I_{\pi^{*}(A)}|\mathcal{F}_{r,n}] \E[I_{\pi'}(A)|\mathcal{F}_{r,n}]]\\
& \qquad =n!\sum_{\star, \star} \E\left[\PP\left(X_{\pi^{*}}\leq 1+\frac{a}{n}|\mathcal{F}_{r,n}\right) \PP\left(X_{\pi'}\leq 1+\frac{a}{n}|\mathcal{F}_{r,n}\right)\right]\,,
\eea\eeq
where $\pi^{*} \in \Sigma_n$ and $\sum_{\star, \star}$ stands for summation over
$$\pi'\in \Sigma_n, (\pi^{*},\pi')\in \Sigma_{n,r}  : 1 \leq \pi^{*}\wedge \pi' \leq n-2.$$
We split this sum into two parts: the first contribution will stem from paths $\pi'$ which share less than $2r$ edges with $\pi^{*}$, in which case $\pi'$ and $\pi^{*}$ are almost independent when n tends to $+\infty$; the second contribution will come from the (fewer) paths which are more correlated with $\pi^{*}$. Precisely, we write:
\beq\bea \label{oooh}
\eqref{estimate3}=& n!\sum_{\star, \star, 1} \E\left[\PP\left(X_{\pi^{*}}\leq 1+\frac{a}{n}|\mathcal{F}_{r,n}\right) \PP\left(X_{\pi'}\leq 1+\frac{a}{n}|\mathcal{F}_{r,n}\right)\right]\\
&\qquad +n!\sum_{\star, \star, 2} \E\left[\PP\left(X_{\pi^{*}}\leq 1+\frac{a}{n}|\mathcal{F}_{r,n}\right) \PP\left(X_{\pi'}\leq 1+\frac{a}{n}|\mathcal{F}_{r,n}\right)\right]
\eea\eeq
while $\sum_{\star, \star, 1}$ denotes summation over
$$\pi'\in \Sigma_n, (\pi^{*},\pi')\in \Sigma_{n,r}  : 1 \leq \pi^{*}\wedge \pi' \leq 2r\,,$$
whereas   $\sum_{\star, \star, 2}$ stands for summation over
$$\pi'\in \Sigma_n, (\pi^{*},\pi')\in \Sigma_{n,r}  : 2r+1 \leq \pi^{*}\wedge \pi' \leq n-2.$$
We now proceed to estimate these two sums: in the first case we will exploit the fact that the involved paths are almost independent. To see how this goes, let
\beq \bea
C_{r,n,\pi'} &\defi
\Big\{e=(u,v)\in E_n, \min \{d(u,\boldsymbol{0}),d(v,\boldsymbol{0})\}\in \left[0,r\right) \cup \left[n-r,n\right),\\
& \hspace{6cm} \text{e is a common edge of $\pi'$ and $\pi^{*}$}\Big\},
\eea \eeq
and denote by $\#C \defi |C_{r,n,\pi'}|$ the cardinality of this set. We now make the following observations:
\begin{itemize}
\item $\#C = 0$ (i.e. $C_{r,n,\pi'}=\emptyset$) implies that $\pi'$ and $\pi^{*}$ are, conditionally upon $\F_{r,n}$, independent.
\item If $\#C >0$, by positivity of exponentials,
\beq\bea
\PP\left(X_{\pi'}\leq 1+\frac{a}{n}|\mathcal{F}_{r,n}\right)&\leq \PP\left(X_{\pi'}-\sum_{e\in C_{r,n,\pi'} } \xi_{e}\leq 1+\frac{a}{n}\Bigg| \mathcal{F}_{r,n}\right)\\
&  = \PP\left(X_{n-\#C}\leq 1+\frac{a}{n}\Big|\mathcal{F}_{r,n}\right)\,,
 \eea\eeq
where $X_{n-\#C}$ is a Gamma$(n-\#C,1)$-distributed random variable which is, conditionally upon $\F_{r,n}$, independent of $X_{\pi^{*}}$.
\end{itemize}
Altogether,
\beq\bea
&n!\sum_{\star, \star, 1} \E\left[\PP\left(X_{\pi^{*}}\leq 1+\frac{a}{n}|\mathcal{F}_{r,n}\right) \PP\left(X_{\pi'}\leq 1+\frac{a}{n}|\mathcal{F}_{r,n}\right)\right]\\
& \hspace{4cm} \leq n!\PP\left(X_{\pi^{*}}\leq 1+\frac{a}{n}\right)\sum_{\star, \star, 1} \PP\left(X_{n-\#C}\leq 1+\frac{a}{n}\right)\,.
\eea\eeq
Convergence of the intensity functions \eqref{intens},  implies that the first term $n!\PP\left(X_{\pi_1}\leq 1+\frac{a}{n}\right)$ converges; in particular, it remains bounded as $n\to \infty$. It therefore suffices to prove that $\sum_{\star, \star, 1} \PP\left(X_{n-\#C}\leq 1+\frac{a}{n}\right)$ tends to 0 in the double limit. To see this, denote by $f(n, k, r)$ the number of paths $\pi'$ that share precisely $k$ edges ($1 \leq k\leq n-2$) with $\pi^{*}$ and with $(\pi',\pi^{*})\in \Sigma_{n,r}$. We then have:
\beq\bea \label{suffices}
\sum_{\star, \star, 1} \PP\left(X_{n-\#C}\leq 1+\frac{a}{n}\right)&=\sum_{k=1}^{2r}f(n, k, r)\PP\left(X_{n-\#C}\leq 1+\frac{a}{n}\right) \\
& \leq \sum_{k=1}^{2r}f(n, k)\PP\left(X_{n-\#C}\leq 1+\frac{a}{n}\right)\,,
\eea\eeq
where $f(n,k)$ is the the number of paths $\pi'$ that share precisely $k\geq 1$ edges with $\pi^{*}$. By the tail-estimates from Lemma  \ref{tail},
\beq \bea \label{proba}
\PP\left(X_{n-\#C}\leq 1+\frac{a}{n}\right) & \leq \frac{\kappa_a}{(n-\#C)!}  \leq \frac{\kappa_a}{(n-k+1)!}\,.
\eea \eeq
The second inequality holds since two paths in $\Sigma_{n,r}$ must share an
edge in the {\it complement} of ${C_{r,n,\pi'}}$. Using \eqref{proba} and Proposition \ref{path_counting} we obtain
\beq\bea
\eqref{suffices} \leq \kappa_a \sum_{k=1}^{2r}\frac{(n-k)!(k+1)}{(n-k+1)!}\,,
\eea\eeq
which vanishes as $n \rightarrow \infty$: the first sum in \eqref{oooh} therefore yields a vanishing contribution.  As for the second sum, by Cauchy-Schwarz,
\beq \label{basta}
n!\sum_{\star, \star, 2}\E[\E[I_{\pi^{*}(A)}|\mathcal{F}_{r,n}] \E[I_{\pi'}(A)|\mathcal{F}_{r,n}]] \leq n!\sum_{\star, \star, 2}\E\left[\PP\left(X_{\pi'}\leq 1+\frac{a}{n}|\mathcal{F}_{r,n}\right)^2 \right]\,.
\eeq
By the tail-estimates from Lemma \ref{tail}, for the expectation on the r.h.s. above it holds
\beq \bea \label{basta2}
& \E\left[\PP\left(X_{\pi'}\leq 1+\frac{a}{n}|\mathcal{F}_{r,n}\right)^2 \right] \\
& \qquad =  \int_{0}^{1+\frac{a}{n}} {\left(1+K(1+\frac{a}{n}-x,n-2r) \right)}^2\frac{e^{-2(1+\frac{a}{n})+x}{(1+\frac{a}{n}-x)}^{2n-4r}x^{2r-1}}{{(n-2r)!}^2(2r-1)!}dx \\
& \qquad \leq \frac{\kappa_a}{{(n-2r)!}^2(2r-1)!}\int_{0}^{1+\frac{a}{n}}{\left(1+\frac{a}{n}-x \right)}^{2n-4r}x^{2r-1}dx.\\
\eea \eeq
Integration by parts then yields
\beq \bea \label{basta5}
& \int_{0}^{1+\frac{a}{n}}{\left(1+\frac{a}{n}-x\right)}^{2n-4r}x^{2r-1}dx \qquad \leq \kappa_a \frac{(2n-4r)!(2r-1)!}{(2n-2r)!}\,,
\eea \eeq
Using \eqref{basta2} and \eqref{basta5} we get
\beq \bea \label{basta3}
\eqref{basta} &\leq \kappa_a \sum_{k=2r+1}^{n-2} \frac{f(n, k, r)}{(n-2r)!}  \frac{n!(2n-4r)!}{(n-2r)!(2n-2r)!}\,.
\eea\eeq
It clearly holds that
\beq
\frac{n!(2n-4r)!}{{(n-2r)!}(2n-2r)!}\leq 1,
\eeq
hence
\beq\bea \label{basta4}
\eqref{basta3} & \leq \sum_{k=2r+1}^{n-2}  \frac{f(n, k, r)}{{(n-2r)!}} \\
& =  \left( \sum_{k=2r+1}^{2r+7}+\sum_{k=2r+8}^{\mathfrak{n_{e}}}+\sum_{k=\mathfrak{n_{e}}+1}^{n-2} \right)
\frac{f(n,k,r)
}{{(n-2r)!}} \\
& =: (A)+ (B)+ (C),
\eea\eeq
say. By Proposition \ref{path_counting}, and worst-case estimates, the following upperbounds hold:
\beq \bea
(A)  & \leq \sum_{k=2r+1}^{2r+7} \frac{(k+1)(n-k)!}{{(n-2r)!}}  \leq \kappa_a \frac{7(2r+8)(n-2r-1)!}{(n-2r)!} \\
(B) &  \leq \sum_{k=2r+8}^{\mathfrak{n_{e}}} \frac{ n^6 (n-k)! }{{(n-2r)!}} \leq n^6\sum_{k=2r+8}^{\mathfrak{n_{e}}}\frac{(n-k)!}{{(n-2r)!}} \leq n^7\frac{(n-2r-8)!}{{(n-2r)!}} \\
(C) & \leq \sum_{k=\mathfrak{n_{e}}+1}^{n-2} \frac{{(2n^{7/8 })}^{n-k} (n-k+1)}{{(n-2r)!}} \leq \frac{n^2{(2n^{7/8})}^{5e(n+3)^{2/3}}}{(n-2r)!}\,.
\eea \eeq
All three terms are clearly vanishing in the limit $n\to \infty$. This implies that the second sum in \eqref{oooh} yields no contribution, and the proof  of Lemma \ref{tre} is thus concluded.
\end{proof}

\section*{Appendix: the conditional Chein-Stein method}
All random variables in the course of the proof are defined on the same probability space $(\Omega, \mathscr{F}, \PP)$. Let $\F\subset \mathscr{F}$ be a sigma algebra, $I$ is a finite (deterministic) set, and $(X_i)_{i\in I}$ a family of Bernoulli random variables. We set
\[
W\defi \sum_{i\in I} X_i, \qquad \lambda \defi \sum_{i\in I}\E(X_i|\F )\,.
\]
Since the claim is trivial for $\lambda = 0$ we assume $\lambda>0$ from here onwards. Additionally we denote by $\widehat{W}$ a random variable which is, conditionally upon $\F$, Poi$(\la)$-distributed, i.e.
\beq
\PP( \widehat{W} = k | \F )(\w) = \frac{\lambda(\w)^k}{k!} e^{-\lambda(\w)}.
\eeq
(To lighten notation, we will omit henceforth the $\w$-dependence). Assume to be given a bounded, $\F$-measurable (possibly random) real-valued function $f$ which satisfies $$\E(f(\widehat{W})|\F)=0,$$ and define $g_f: \N \to \R$ by
\beq \label{eq1}
 g_f(0) \defi 0, \quad g_f(n) \defi \frac{(n-1)!}{{\lambda}^n}\sum\limits_{k=0}^{n-1}\frac{f(k)\lambda^k}{k!} \quad n>0\,.
\eeq
We claim that $g_f$ is $\F$-measurable, bounded, and satisfies the following identities:
\beq\bea \label{g_f_id1}
 f(n)=\lambda g_f(n+1)- ng_f(n), \qquad  n\geq0\,,
\eea\eeq
and
\beq\bea \label{g_f_id2}
g_f(n)=-\frac{(n-1)!}{{\lambda}^n}\sum\limits_{k=n}^{\infty}\frac{f(k)\lambda^k}{k!} \qquad n > 0.
\eea\eeq
Measurability and first identity follow steadily from the definition. The second identity follows from the fact that
$\E(f(\widehat{W})|\F)=0$, whereas boundedness follows from the integral representation of the Taylor rest-term of the exponential function:
\beq \label{eq2}
\mid g_f(n)\mid \leq \frac{(n-1)! \max_{k \in \N}\mid f(k)\mid}{\lambda^n}\int\limits_{0}^{\lambda} \frac{t^{n-1}}{(n-1)!}e^t dt \leq \frac{\max_{k \in \N}\mid f(k)\mid e^\lambda}{n}.
\eeq
Let now $A \subset \N_0$, and consider the function
\beq\label{f}
f_{A,\lambda}(n) \defi \1_{n \in A}-\PP(\widehat{W}\in A|\F), \quad n \in \N.
\eeq
This is clearly a bounded, $\F$-measurable function which satisfies $\E(f_{A,\lambda}(\widehat{W})|\F)=0$. Therefore, by the above and in particular \eqref{g_f_id1}, there exists a bounded $\F$-measurable function, denoted by $g_{A,\lambda}$, which satisfies
\beq
 \1_{n \in A}-\PP(\widehat{W}\in A|\F)=\lambda g_{A,\lambda}(n+1)-n g_{A,\lambda}(n) ,
\eeq
almost surely for any $n \in \N$. It follows that
\beq
 \1_{W \in A}-\PP(\widehat{W}\in A|\F)=\lambda g_{A,\lambda}(W+1)-W g_{A,\lambda}(W).
\eeq
Taking conditional expectations thus yields
\beq\bea\label{eq3}
\PP({W \in A}|\F)-\PP(\widehat{W}\in A|\F) &= \lambda \E(g_{A,\lambda}(W+1)|\F)-\E(W g_{A,\lambda}(W)|\F)\\
&= \sum\limits_{i \in I} \E(X_i|\F)\E(g_{A,\lambda}(W+1)|\F)-\E(X_i g_{A,\lambda}(W)|\F).
\eea\eeq
Consider now the random subset
\[
N_i \defi \{j\in I\setminus{\{i\}}:\; X_j \; \text{and}\; X_i\; \text{ are not conditionally independent given}\; \F \},
\]
and denote by $S^{(i)}$ a random variable which is distributed like $\sum\limits_{j \in N_i}{X_j}$ conditionally upon $\F$ and $\{X_i=1\}$, i.e.
\beq
\PP(S^{(i)} = k |\F) = \PP\left(\sum\limits_{j \in N_i}{X_j} = k, X_i = 1 \Big|\F \right) \Big/  \PP( X_i = 1\big|\F) \,.
\eeq
if $\PP( X_i = 1|\F)>0$, and arbitrarily defined otherwise.

\noindent We remark that $X_i$ and $(X_j)_{j\in ( N_i \cup \{i\})^c}$ are conditionally on $\F$ independent. Therefore
\beq\bea \label{forheavensake}
\E(X_i g_{A,\lambda}(W)|\F)=\PP(X_i=1|\F)\E\left[ g_{A,\lambda}\left(1+S^{(i)}+\sum\limits_{j \in I\setminus (N_i \cup \{i\})}{X_j} \right)\Bigg| \F\right],
\eea\eeq
since $X_i$ and $X_j$ are conditionally independent given $\F$. Plugging this into the r.h.s. of \eqref{eq3} yields
\beq\bea
& \PP({W \in A}|\F)-\PP(\widehat{W}\in A|\F)  \\
& \hspace{2cm} = \sum\limits_{i \in I} \E(X_i|\F)\E\left[g_{A,\lambda}(1+W) - g_{A,\lambda}(1+S^{(i)}+\sum\limits_{j \in I\setminus (N_i \cup \{i\})}{X_j})\Bigg| \F\right].
\eea\eeq
Set now
\beq \label{m}
M\defi \sup\{|g_{A,\lambda}(n+1)-g_{A,\lambda}(n)|: n\in \N_0\}\,.
\eeq
(Notice that $M$ is $\F$-measurable). By the triangle inequality, and worstcase-scenario,
\beq\bea
& \mid \PP({W \in A}|\F)-\PP(\widehat{W}\in A|\F) \mid \leq M \sum\limits_{i \in I}  \E(X_i|\F)  \E(X_i + S^{(i)}+ \sum\limits_{j\in N_i} X_j   | \F) \\
& \hspace{2cm} = M \sum\limits_{i \in I} \left(E(X_i|\F)^2+\sum\limits_{j \in N_i} \left(\E(X_jX_i|\F)+\E(X_j|\F)E(X_i|\F) \right)\right)\,.
\eea\eeq
It remains to prove that $M\leq1$. To this end we observe that additivity of $g_{.,\lambda}$ is inherited from $f_{.,\lambda}$, hence
\beq\label{add}
g_{A,\lambda} = \sum\limits_{j\in A} g_{\{j\},\lambda}\,.
\eeq
Furthermore,
\beq \label{zero}
\sum\limits_{j=0}^{\infty}g_{\{j\},\lambda}(n+1)-g_{\{j\},\lambda}(n)=0,
\eeq
since
\beq\bea
\sum\limits_{j=0}^{\infty}g_{\{j\},\lambda}(n) &\stackrel{\eqref{add}}{=} g_{\N_0,\lambda}(n) = 0 \quad \forall n\in \N,
\eea\eeq
because $f_{\N_0,\lambda}$ is the zero function. Therefore, for any $A\subset \N_0$,
\beq\label{technical_stuff1}
|g_{A,\lambda}(n+1)-g_{A,\lambda}(n)| \leq \sum\limits_{j=0}^\infty (g_{\{j\},\lambda}(n+1)-g_{\{j\},\lambda}(n))^{+}.
\eeq
By \eqref{eq1}, the definition of $f$ and elementary computations we have, for $0<n\leq j$, that
\beq\label{exp}
g_{\{j\},\lambda}(n)=-\PP(\widehat{W}=j|\F) \sum\limits_{l=0}^{n-1} \frac{(n-1)!}{\lambda^{l+1}(n-1-l)!}.
\eeq
This implies in particular that $g_{\{j\},\lambda}(n)$ is decreasing in $n$ on $[0,j]$, hence all summands $j\geq n+1$ in \eqref{technical_stuff1}  vanish. On the other hand, by \eqref{g_f_id2}, again the definition of $f$ and elementary computations we have for $n>j$
\beq\label{exp2}
g_{\{j\},\lambda}(n)=\PP(\widehat{W}=j|\F) \sum\limits_{l=0}^{\infty} \frac{\lambda^{l}(n-1)!}{(n+l)!}.
\eeq
Since this is also decreasing in $n$, it follows that $j=n$ is the only non-zero summand in \eqref{technical_stuff1}. All in all,
\beq  \label{less_zero}
M= \sup_{n\in\N}\mid g_{A,\lambda}(n+1)-g_{A,\lambda}(n)\mid  \leq \sup_{n\in\N} \mid g_{\{n\},\lambda}(n+1)-g_{\{n\},\lambda}(n)\mid.
\eeq
Now, for $n>0$, by \eqref{exp} and \eqref{exp2},
\beq\bea \label{less1}
& \mid g_{\{n\},\lambda}(n+1)-g_{\{n\},\lambda}(n)\mid = \\
& \hspace{2cm} =\frac{\lambda^n e^{-\lambda}}{n!}\left(  \sum\limits_{l=0}^{\infty} \frac{\lambda^{l}(n-1)!}{(n+l)!}+\sum\limits_{l=0}^{n-1} \frac{(n-1)!}{\lambda^{l+1}(n-1-l)!}\right)\\
& \hspace{2cm} =\frac{e^{-\lambda}}{n}\left(\sum\limits_{l=n}^{\infty}\frac{{\lambda}^l}{l!}+\sum\limits_{l=0}^{n-1}\frac{{\lambda}^{l} }{l!}\right) = \frac{1}{n}\leq1.
\eea\eeq
On the other hand, for $n=0$,
\beq \label{less2}
\mid g_{\{0\},\lambda}(1)-g_{\{0\},\lambda}(0)\mid =  \frac{1}{\lambda}(1-e^{-\lambda})\leq 1\,,
\eeq
by Taylor estimate. Using \eqref{less1} and \eqref{less2} in \eqref{less_zero} shows that $M\leq1$ as claimed, and concludes the proof of the conditional Chen-Stein method. \\

\hfill $\square$

\end{document}